\documentclass[12pt]{article}
\usepackage{mathrsfs}
\usepackage{amssymb}
\usepackage{color}

\newcounter{RomanNumber}
\newcommand{\MyRoman}[1]{\setcounter{RomanNumber}{#1}\Roman{RomanNumber}}

\usepackage{amsthm}
\usepackage{amsmath}
\usepackage{graphicx}

 \newtheorem{thm}{Theorem}[section]
 \newtheorem{cor}[thm]{Corollary}
 \newtheorem{lem}[thm]{Lemma}
 
 \theoremstyle{definition}
 
 \newtheorem{rem}[thm]{Remark}
 \numberwithin{equation}{section}

\theoremstyle{definition}

\theoremstyle{remark}

\begin{document}
\title{On Scaling Invariance and Type-\MyRoman{1} Singularities for the Compressible Navier-Stokes Equations}
\author{
Zhen Lei
\footnote{School of Mathematical Sciences; LMNS and Shanghai
 Key Laboratory for Contemporary Applied Mathematics, Fudan University, Shanghai 200433, P. R.China. {\it Email:
 leizhn@gmail.com}}\and Zhouping Xin\footnote{The Institute of Mathematical Sciences and Department of Mathematics,
The Chinese University of Hong Kong, Shatin, N.T., Hong Kong.
{\it Email: zpxin@ims.cuhk.edu.hk}
}
 }
\date{\today}
\maketitle

\begin{abstract}
We find a new scaling invariance of the barotropic compressible Navier-Stokes equations. Then it is shown that type \MyRoman{1} singularities of solutions with
$$\limsup_{t \nearrow T}|{\rm div} u(t, x)|(T - t) \leq \kappa,$$
can never happen at time $T$ for all adiabatic number $\gamma \geq 1$. Here $\kappa > 0$ doesn't depend on the initial data. This is achieved by proving the regularity of solutions under
$$\rho(t, x) \leq \frac{M}{(T - t)^\kappa},\quad M < \infty.$$
This new scaling invariance also motivates us to construct an explicit type \MyRoman{2} blowup solution for $\gamma > 1$.
\end{abstract}

\maketitle





\section{Introduction}

We study the Cauchy problem for the compressible
Navier-Stokes equations
\begin{equation}\label{CNS}
\begin{cases}
\partial_t\rho + \nabla\cdot (\rho u) = 0,\\[-4mm]\\
\partial_t(\rho u) + \nabla\cdot(\rho u\otimes u) + \nabla P(\rho)
  = \mu\Delta u + (\lambda + \mu)\nabla\nabla\cdot u,\\[-4mm]\\
P(\rho) = A\rho^\gamma,
\end{cases}
\end{equation}
 which govern the motion of a compressible viscous and polytropic
Newtonian fluid.
As usual, $\rho: \mathbb{R}_+ \times \Omega \rightarrow
\mathbb{R}_+$ denotes the density of the fluid flows, $u: \mathbb{R}_+ \times
\Omega \rightarrow \mathbb{R}^n$ is the velocity field and $P(\rho) = A\rho^\gamma$ the
scalar pressure. Moreover, $\Omega \subseteq \mathbb{R}^n$ is a smooth bounded domain or the whole space with $n$ ($= 2\ {\rm or}\ 3$) being the space dimension, the constant $A > 0$, the adiabatic number $\gamma \geq 1$
and the viscosity constants $\lambda$ and $\mu$ satisfy the physical constraint
\begin{equation}\label{phycons}
\mu > 0,\ \ n\lambda + 2\mu > 0.
\end{equation}
If $\lambda = \mu = 0$ in \eqref{CNS}, one recovers the compressible Euler equations.

The main results of this paper are that there is a new scaling law for \eqref{CNS} and certain type \MyRoman{1}  singularities can be excluded.
The new scaling invariance of the compressible Navier-Stokes equations is stated as a theorem below which is our key observation. Here we do not pursue the exact meaning of "solution".
\begin{thm}\label{ScalingI}
Let $(\rho, u)$ be a solution to the compressible Navier-Stokes equations \eqref{CNS} and $\kappa > 0$. Define
\begin{equation}\label{Scaling}
\begin{cases}
\rho^\kappa(t, x) = \kappa^{\frac{1}{\gamma}}\rho(\kappa t, \kappa^{\frac{\gamma + 1}{2\gamma}}x),\\[-4mm]\\
u^\kappa(t, x) = \kappa^{\frac{\gamma - 1}{2\gamma}}u(\kappa t, \kappa^{\frac{\gamma + 1}{2\gamma}}x).
\end{cases}
\end{equation}
Then $(\rho^\kappa, u^\kappa)$ is also a solution to \eqref{CNS} for every $\kappa > 0$.
\end{thm}
While such a scaling invariance can be verified directly, we believe that it is interesting in itself. As is well-known, the natural scaling invariance has played one of the most essential roles in the study of incompressible Navier-Stokes equations at least from a heuristical point of view (for instance, the famous Caffarelli-Kohn-Nirenberg theory \cite{CKN:1}, the small data global existence type results in critical spaces \cite{Chemin, FK, Kato, Planchon, KTa, LL}, etc.). Certain type of asymptotic scalings are also one of the backbones in the study of compressible Navier-Stokes equations (for instance, the work of Dachin \cite{Danchin} and so on).  The formula \eqref{Scaling} gives the first precise scaling laws for the barotropic compressible Navier-Stokes equations. Based on the scaling, one may formally assign the dimensions of space-time variables and unknowns as follows:
Each space variable $x_j$ has dimension $+ 1$, the time variable $t$  has dimension $+ \frac{2\gamma}{\gamma + 1}$, the density function $\rho$ has dimension $- \frac{2}{\gamma + 1}$
and the velocity vector $u$ has dimension $- \frac{\gamma - 1}{\gamma + 1}$.

 At a heuristical level, the scaling invariance given in \eqref{Scaling} suggests that local strong solutions with certain blowup rate will stop the formation of true singularities. It also suggests us a way to construct finite time singular solutions with special forms. We will confirm these intuitive ideas and we hope it could be useful elsewhere.  We would like to mention that for type \MyRoman{1} solutions of the Navier-Stokes equations in the incompressible case, a significant progress has been achieved recently for the axi-symmetric case \cite{CSTY, KNSS, CSYT, LZ}.

For the compressible Navier-Stokes equations \eqref{CNS}, there have been extensive literatures on the existence and finite time singularity of solutions. In particular, in the absence of vacuum, a uniqueness result was obtained
by Serrin \cite{Serrin} and the local existence results were given by Nash \cite{Nash} and Itaya \cite{Itaya}. Later on, non-vacuum small perturbations of a uniform non-vacuum constant state have been shown to exist globally in time and remain smooth in any
space dimensions \cite{MN, MN2, Hoff1, Hoff2, Hoff3, Danchin, CMZ}. Those works are based on the dissipative nature of the system. In the presence of vacuum, the system is strongly degenerate and the problem becomes extremely completed. A breakthrough was made by Lions \cite{Lions} where global existence of weak solutions with finite energy was established for adiabatic number $\gamma > \frac{9}{5}$ (see also Feireisl, Novotn${\rm \acute{y}}$ and Petzeltov${\rm \acute{a}}$ \cite{Fei2} for the case of $\gamma > \frac{3}{2}$, and Jiang and Zhang \cite{JZ1, JZ2} for the case of $\gamma > 1$ under symmetry). In recent years, there are tremendous works on the existence and uniqueness of classical solutions in the presence of vacuum, see for instance, \cite{Fei1, Des, Solo, CCK-1, CCK-2, HLX, Luo} and the references therein. In particular, when the initial total energy is sufficiently small, the well-posedness holds globally in time \cite{HLX}. For the formation of singularities,  Xin \cite{Xin} showed that in the case that the initial density has a compact support, any smooth solution with $u \in C^1([0, T] : H^s(\mathbb{R}^n))(s > [n/2]+2)$ to the Cauchy problem of the
full compressible Navier-Stokes system without heat conduction blows up in finite time
for any space dimension $n \geq 1$ (see also \cite{XY, HXin} for further results). Li, Wang and Xin \cite{LWX} even proved that the classical solution with finite energy does not exist in the inhomogeneous Sobolev space for any short time under some natural assumptions on initial data near the vacuum.
For the compressible Euler Equations, finite time singularities of solutions have been shown in \cite{Sideris} in the absence of vacuum and \cite{LDZ} in the presence of vacuum. There are also many works on blowup criterions on the compressible Navier-Stokes equations, for instance, see \cite{FJO, JW, HLX2, HZ1, HZ2, SWZ, WZ} and the references therein. Of all those works, \cite{Hoff2, HLX2, SWZ, CCK-2, WZ} are the most relevant for our work and will be discussed in detail below.

We first impose the boundary condition when $\Omega$ is a bounded  domain with a smooth boundary
\begin{equation}\nonumber
u = 0\quad {\rm on}\ \partial\Omega.
\end{equation}
We assume that the initial data satisfy
\begin{equation}\label{data}
\begin{cases}
\rho(0, \cdot) = \rho_0(\cdot) \in L^1 \cap H^s,\\[-4mm]\\
u_0 \in H_0^1\cap H^s,
\end{cases}
s \geq 3.
\end{equation}
It is clear that \eqref{data} implies $\rho_0(\cdot) \in L^\infty$ and $\|\sqrt{\rho_0}u_0\|_{L^2} < \infty$.
Here and in what follows we use $H^s$ and $W^{s, q}$ to denote the usual Sobolev space. For $f \in H_0^1$,  if $\Omega$ is a smooth bounded domain, we meant that $f = 0$ on the boundary in the sense of trace. Due to \cite{CCK-2}, there exists a unique local strong solution to the compressible Navier-Stokes equations with initial data \eqref{data} under certain compatibility conditions on the initial data. Moreover, the local strong solution $(\rho, u)$ satisfies
\begin{equation}\label{solu-1}
\begin{cases}
\rho \in L^1\cap C([0, T), W^{1, q}),\quad u \in C([0, T), L^{6}),\quad u(t, \cdot)|_{\partial\Omega} = 0,\\[-4mm]\\
\nabla u \in C([0, T),  H^1),\quad \int_0^t\|\nabla^2u\|_{L^q}^2ds < \infty\ (\forall\ t < T),
\end{cases}
3 < q < 6,
\end{equation}
and the solution is regular at time $T$ if
\begin{equation}\label{criterion}
\sup_{0 \leq t < T}(\|\nabla \rho\|_{L^q} + \|\nabla u\|_{L^2}) < \infty.
 \end{equation}
 We remark that throughout this paper, we say the solution $(\rho, u)$ is regularity at time $T$ if the continuity in time holds at $t = T$ in \eqref{solu-1}.

For the local strong solution in \eqref{solu-1}, Huang and Xin \cite{HZ2} and Huang, Li and Xin \cite{HZ1} proved the regularity at time $T$ if
$$\lim_{t \nearrow T}\int_0^t\|\nabla u + (\nabla u)^T\|_{L^\infty}ds < \infty.$$
Under extra constraint $\lambda < 7\mu$, Sun, Wang and Zhang \cite{SWZ} (under $\rho_0 > 0$) and independently, Huang, Li and Xin \cite{HLX2} proved the regularity at time $T$ if
$$\sup_{0 \leq t < T}\|\rho\|_{L^\infty} < \infty.$$ Wen and Zhu \cite{WZ} weakened the constraint to be $\lambda < \frac{29\mu}{7}$ and improved the result in terms of $\|\rho\|_{L^q}$ for a sufficiently large $q$. We will use some ideas and calculations in \cite{Hoff2, HLX2, HZ2, SWZ, WZ} and follow the framework of \cite{HLX2, HZ2, SWZ}.

The second result of this paper is the following theorem which asserts that the solution in \eqref{solu-1} is smooth at time $T$ if it is type \MyRoman{1}. Due to the strong degeneracy of the parabolic nature of the system, the index $\kappa$ is not \textit{optimal} in our theorems.  Based on the above scaling invariance \eqref{Scaling}, we conjecture that the optimal index might be $\frac{1}{\gamma}$, which seems a fantastic challenge to us.

\begin{thm}[No type \MyRoman{1} Singularities]\label{TypeI}
Let $\gamma \geq 1$, $n = 3$, $A > 0$, $T > 0$, $K \leq 7$, $\lambda < K\mu$ be constants and satisfy the physical constraint \eqref{phycons}. Let $p \in (3, 6)$ be determined in Lemma \ref{elementarylem} and $\kappa$ satisfy
\begin{equation}\label{Constraint}
\kappa < \min\big\{\frac{1}{\gamma + 3},\quad \frac{1}{3\gamma},\quad \frac{p - 3}{p + 1}\big\}.
\end{equation}
Then $(\rho, u)$ in \eqref{solu-1} is regular at time $T$ provided that
\begin{equation}\label{UPB-1}
|\nabla\cdot u(t, x)| \leq \frac{\kappa}{T - t},\quad {\rm as}\ t \nearrow T.
\end{equation}
If $K = 2$, $p$ can be taken to be 4.
\end{thm}

The proof of Theorem \ref{TypeI} will be presented in Section 2 by assuming the validity of Theorem \ref{main}.
\begin{rem}
The constant $p$ determined in Lemma \ref{elementarylem} depends only on the physical viscosity constants $\lambda$ and $\mu$, but not on solutions and other constants. With extra complicated calculations, one may improve $K \leq 7$ as $K \leq \frac{29}{7}$ (see Wen and Zhu \cite{WZ} for that purpose).
\end{rem}

\begin{rem}
For $\gamma > 1$, it is clear that the following energy law holds for the solutions in Theorem \ref{TypeI}
\begin{eqnarray}\label{basicen}
&&\frac{1}{2}\int \big(\rho|u|^2 + \frac{A}{\gamma - 1}\rho^\gamma\big)dx + \int_0^t\int\big(\mu|\nabla u|^2 + (\lambda + \mu)|\nabla\cdot u|^2\big)dx\\\nonumber
 &&=  \frac{1}{2}\int \big(\rho_0|u_0|^2 + \frac{A}{\gamma - 1}\rho_0^\gamma\big)dx,
\end{eqnarray}
which will be frequently used throughout this paper. For instance, we often use the following
$$\sqrt{\rho}u \in L^\infty(L^2_x),\quad \rho \in L^\infty(L^\gamma_x),\quad \nabla u \in L^2(L^2_x).$$
In the whole space case, vacuum might be very common since $\rho \in L^\gamma$. In the bounded domain case, vacuum is also allowed to appear.
\end{rem}

\begin{rem}
When $\gamma = 1$, which corresponds to the isothermal process, instead of using the basic energy law \eqref{basicen}, we will use
the conservation of mass $$\int \rho(t, x)dx = \int \rho_0(x)dx$$ and the following alternative
\begin{eqnarray}\label{basicen-1}
&&\frac{1}{2}\int\rho|u|^2 dx + \mu\int_0^t\|\nabla u\|_{L^2}^2ds + \frac{\lambda + \mu}{2}\int_0^t\|\nabla\cdot u\|_{L^2}^2ds\\\nonumber
&&= \frac{1}{2}\int\rho_0|u_0|^2 dx +  A\int_0^t\int \rho\nabla\cdot u dxds - \frac{\lambda + \mu}{2}\int_0^t\|\nabla\cdot u\|_{L^2}^2ds\\\nonumber
&&\leq \frac{A^2\|\rho\|_{L^1}}{2(\lambda + \mu)}\int_0^t\|\rho\|_{L^\infty}ds.
\end{eqnarray}
Under the constraint on $\rho$ in Theorem \ref{main}, the right hand side of \eqref{basicen-1} is uniformly bounded on $t \in [0, T)$. Hence the case of $\gamma = 1$ can be treated exactly in the same way as $\gamma > 1$, with even a simpler calculation.
So in what follows, we will only focus on the case when $\gamma > 1$.
\end{rem}

To obtain Theorem \ref{TypeI}, we will prove the following stronger result, which implies that certain possible concentration of density is removable and won't lead to the formulation of finite-time singularities.

\begin{thm}\label{main}
Let $M > 0$ and all other constants be given in Theorem \ref{TypeI} and satisfy the same constraints in Theorem \ref{TypeI}.
Then $(\rho, u)$ in \eqref{solu-1} is regular at time $T$  provided that
\begin{equation}\label{UPB}
\rho(t, x) \leq\frac{M}{(T - t)^{\kappa}},\quad 0 \leq t < T.
\end{equation}
If $K = 2$, $p$ can be taken to be 4.
\end{thm}

The proof of Theorem \ref{main} will be presented in Section 3, Section 4 and Section 5.

The third result of this article is to construct the following explicit solution to the compressible Euler and Navier-Stokes equations  which blows up at any given finite time $T > 0$.

\begin{thm}\label{Exam}
For any $T > 0$, the following solution pair
$$\rho(t, x) = C_n^{\frac{1}{\gamma - 1}}\Big(\frac{|x|}{T - t}\Big)^{\frac{2}{\gamma - 1}},\quad u(t, x) = - \frac{2x}{[n(\gamma - 1) + 2](T - t)}$$
solves the compressible  Navier-Stokes equations \eqref{CNS} for all $\gamma > 1$, where constants $C_n$ are given by
\begin{equation}\label{Const}
C_2 = \frac{(\gamma - 1)^2}{2A\gamma^3},\quad C_3 = \frac{3(\gamma - 1)^2}{A\gamma(3\gamma - 1)^2}.
\end{equation}
\end{thm}

\begin{rem}\nonumber
For the physical adiabatic number $1 < \gamma < 3$, it is easy to check that there holds
\begin{equation}\nonumber
\rho \in
\begin{cases}
C^\infty([0, T), C^{1 + [\alpha], \alpha - [\alpha]}(\mathbb{R}^n)),\quad {\rm if}\ \alpha > 0\ {\rm is\ not\ an\ integer},\\[-4mm]\\
C^\infty([0, T), C^{\alpha + 1}(\mathbb{R}^n)), \quad {\rm if}\ \alpha > 0\ {\rm is\ an\ odd\ integer},\\[-4mm]\\
C^\infty([0, T), C^{\alpha, 1}(\mathbb{R}^n)), \quad {\rm if}\ \alpha > 0\ {\rm is\ an\ even\ integer}.
\end{cases}
\end{equation}
Here $\alpha = \frac{3 - \gamma}{\gamma - 1}$.
\end{rem}

The construction of the explicit blowup example is inspired by the above dimension analysis and the scaling invariance \eqref{Scaling}. Indeed, we consider self-similar solutions of the following form
\begin{equation}\label{SSS}
\begin{cases}
\rho(t, x) = \frac{1}{(T - t)^{\frac{1}{\gamma}}}\Theta(\frac{x}{(T - t)^{\frac{\gamma + 1}{2\gamma}}}),\\[-4mm]\\
u(t, x) = \frac{1}{(T - t)^{\frac{\gamma - 1}{2\gamma}}}V(\frac{x}{(T - t)^{\frac{\gamma + 1}{2\gamma}}}).
\end{cases}
\end{equation}
One can derive that $(\Theta, V)$ are governed by \eqref{SCNS}. Note that \eqref{SCNS} is still a complicated system of nonlinear differential equations and in general hard to solve. Fortunately, the special structure of \eqref{SCNS} allows us to construct explicit solutions to it, which gives the explicit solution in Theorem \ref{Exam}.  Details are presented in Section 6. We remark that the solution in Theorem \ref{Exam} is constructed in H${\rm \ddot{o}}$lder spaces. It is an interesting question to construct singular solutions to \eqref{SCNS} which live in Sobolev spaces before the blowup time.

\section{Preliminaries and Proof of Theorem \ref{TypeI}}

Let us first give the proof of Theorem \ref{TypeI}, by assuming the validity of Theorem \ref{main}. The proof of Theorem \ref{main} will be postulated to Section 3, Section 4 and Section 5.

\begin{proof}[Proof of Theorem \ref{TypeI}]

 With loss of generality, by using \eqref{UPB-1}, one may assume that
$$|\nabla\cdot u(t, \cdot)| \leq \frac{\kappa}{T - t},\quad 0 \leq t < T.$$
By the continuity equation in \eqref{CNS}, one has
\begin{eqnarray}\nonumber
\rho(t, x) &\leq& \|\rho_0\|_{L^\infty}e^{\int_{0}^t\|\nabla\cdot u(t, \cdot)\|_{L^\infty}ds}\\\nonumber
&\leq& \frac{\|\rho_0\|_{L^\infty}T^\kappa}{(T - t)^\kappa},\quad \forall\ 0 \leq t < T.
\end{eqnarray}
Then the proof of Theorem \ref{TypeI} is a straightforward consequence of Theorem \ref{main}.
\end{proof}

Let $v(t, x)$ be the solution to
\begin{equation}\nonumber
Lv = A\nabla\rho^\gamma,
\end{equation}
where the elliptic operator $L$ is defined by
$$Lu = \mu\Delta u + (\lambda + \mu)\nabla\nabla\cdot u.$$
In the case of a smooth bounded domain, we also impose the boundary condition
$$v = 0\ \ {\rm on}\ \partial\Omega.$$
By standard elliptic estimates (see, for instance, \cite{SWZ} and the references therein), the assumption on $\rho$ in Theorem \eqref{main} and the interpolation inequality, one has
\begin{lem}\label{lem3}
Under the assumption of Theorem \ref{main}, there hold
\begin{equation}\label{4}
\begin{cases}
\|\nabla v(t)\|_{L^q} \lesssim \|\rho^\gamma\|_{L^q} \lesssim \|\rho^\gamma\|_{L^1}^{\frac{1}{q}}(T - t)^{- \kappa\gamma(1 - \frac{1}{q})},\quad 1 < q < \infty,\\[-4mm]\\
\|\nabla v(t)\|_{{\rm BMO}} \lesssim \|\rho^\gamma\|_{L^\infty} \lesssim (T - t)^{- \kappa\gamma},
\end{cases}
\end{equation}
and
\begin{eqnarray}\label{5}
\|\nabla^2 v(t)\|_{L^q} \lesssim \|\rho\|_{L^\infty}^{\gamma - 1}\|\nabla\rho\|_{L^q} \lesssim (T - t)^{- \kappa(\gamma - 1)}\|\nabla\rho\|_{L^q}.
\end{eqnarray}
\end{lem}

The next lemma is well-known and can be found in, for instance, \cite{BW, KT, SWZ}.
\begin{lem}\label{lem1}
Let $v$ be defined above and $q > 3$. Then
\begin{eqnarray}\nonumber
\|\nabla v\|_{L^\infty} \lesssim 1 + \|\rho\|_{L^\infty}^\gamma\ln\big(e + \|\rho\|_{L^\infty}^{\gamma - 1}\|\nabla\rho\|_{L^q}\big).
\end{eqnarray}
\end{lem}
\begin{proof}
Indeed, by \cite{BW, KT, SWZ}, one has
\begin{eqnarray}\nonumber
\|\nabla v\|_{L^\infty} \lesssim 1 + \|\nabla v\|_{L^2} + \|\nabla v\|_{{\rm BMO}}\ln\big(e + \|\nabla^2 v\|_{L^q}\big).
\end{eqnarray}
Then \eqref{4}, \eqref{5} and the basic energy law \eqref{basicen} imply that
\begin{eqnarray}\nonumber
&&\|\nabla v\|_{L^2} + \|\nabla v\|_{{\rm BMO}}\ln\big(e + \|\nabla^2 v\|_{L^q}\big)\\\nonumber
&&\lesssim \|\rho\|_{L^\infty}^{\frac{\gamma}{2}} + \|\rho\|_{L^\infty}^\gamma\ln\big(e + \|\nabla\rho^\gamma\|_{L^q}\big)\\\nonumber
&&\lesssim \|\rho\|_{L^\infty}^\gamma\ln\big(e + \|\rho\|_{L^\infty}^{\gamma - 1}\|\nabla\rho\|_{L^q}\big).
\end{eqnarray}
\end{proof}

At last, let us give an elementary lemma.
\begin{lem}\label{elementarylem}
Let $\lambda < K\mu$, $K \leq 7$ and the physical constraint \eqref{phycons} be satisfied. Then there exists $p \in (3, 6)$ such that the following quadratic form
$$f(X, Y) = 4\mu(X^2 + Y^2) - (\lambda + \mu)R^2Y^2 + 4\mu R Y^2$$
when $R$ is evaluated at $p - 2$, has a lower bound of $X^2 + Y^2$ multiplied by a small positive constant. Moreover, if $K = 2$, then $p$ can be taken to be 4.
\end{lem}
\begin{proof}
With a little bit ambiguity of notation $f$, we consider the following quadratic polynomial for $1 < R < 4$:
\begin{eqnarray}\nonumber
f(R) = 4\mu - (\lambda + \mu)R^2 + 4\mu R.
\end{eqnarray}
Note that $\lambda + \mu > 0$ under \eqref{phycons}. It is clear that $f(2) = 4(2\mu - \lambda)$ and $f(1) = 7\mu - \lambda$. By an elementary analysis, one can conclude that
\begin{eqnarray}\nonumber
f(1+) > 0, &&{\rm if}\ \lambda < K\mu,\quad K\leq 7,\\\nonumber
f(2) > 0, &&{\rm if}\ \lambda < K \mu,\quad K \leq 2.
\end{eqnarray}
Hence, for all $\lambda < K\mu$, $K \leq 7$ satisfying \eqref{phycons}, there exists $p \in (3, 6)$ which might be very close to $3$ such that
\begin{eqnarray}\nonumber
f(p - 2) > 0.
\end{eqnarray}
Moreover, if $K \leq 2$ and \eqref{phycons} holds, then $p$ can be taken to be 4. Then the lemma is proved by rewriting $f(X, Y)$ as
$$f(X, Y) =  4\mu X^2 + Y^2f(R).$$
\end{proof}

\section{Energy Estimates}

In this section we follow the framework in \cite{SWZ, HZ2} and prove two energy estimates. So we will omit those same computations as there. The first one gives the uniform in time bound of the $L^1$ norm of $\rho |u|^p$ under the assumptions in Theorem \ref{main}. The second one is on the estimate of $L^2$ type energy estimate for $\nabla(u - v)$, together with a space-time  estimate of $\nabla^2(u - v)$. Then we give a corollary which asserts that the $L^2$ norm of $\nabla u$ may grow in time at $(T - t)^{- \frac{\kappa\gamma}{2}}$, while $\int_0^T\|\nabla u\|_{L^6}^{1 + \delta}dt$ is still bounded for some $\delta \in (0, 1)$.
From now on we focus on the whole space case. The bounded domain case with smooth boundaries can be treated similarly without any essential difficulty in view of the homogenous Dirichlet boundary condition for $u$ and $v$.

Let us first prove the following lemma.
\begin{lem}\label{energyes}
Under the assumptions in Theorem \ref{main}, there holds
$$\big\|\rho|u|^{p}\big\|_{L^1} + \int_0^T\big\||u|^{\frac{p}{2} - 1}|\nabla u|\big\|_{L^2}^2dt \lesssim 1.$$
\end{lem}
\begin{proof}
Standard estimates as \cite{Hoff2, HLX, SWZ} yield
\begin{eqnarray}\nonumber
&&\frac{1}{p}\frac{d}{dt}\int \rho |u|^p dx \leq A\int \rho^\gamma\nabla\cdot (|u|^{p - 2}u) dx\\\nonumber
&&\quad -\ \int |u|^{p - 2}\big(\mu|\nabla u|^2 + (p - 2)[\mu - (\lambda + \mu)\frac{p - 2}{4}]|\nabla |u||^2\big) dx.
\end{eqnarray}
Set $X^2 = |\nabla u|^2 - |\nabla |u||^2$ and $Y^2 = |\nabla |u||^2$. Then by Lemma \ref{elementarylem}, one has
$$\int |u|^{p - 2}\big(\mu|\nabla u|^2 + (p - 2)[\mu - (\lambda + \mu)\frac{p - 2}{4}]|\nabla |u||^2\big) dx \geq c_0\int |u|^{p - 2}|\nabla u|^2 dx$$
and thus
\begin{eqnarray}\label{1}
\frac{d}{dt}\int \rho |u|^p dx + pc_0\int |u|^{p - 2}|\nabla u|^2 dx \leq Ap\int \rho^\gamma\nabla\cdot (|u|^{p - 2}u) dx.
\end{eqnarray}

Let us treat the right hand side in \eqref{1} as follows:
\begin{eqnarray}\nonumber
&&\int \rho^\gamma\nabla\cdot (|u|^{p - 2}u) dx\\\nonumber
&&\lesssim \int \rho^{\gamma + \frac{1}{p} - \frac{1}{2}}(\rho|u|^{p})^{\frac{1}{2} - \frac{1}{p}}(|u|^{\frac{p}{2} - 1}|\nabla u|)dx\\\nonumber
&&\lesssim \|\rho\|_{L^{p\gamma + 1 - \frac{p}{2}}}^{\gamma + \frac{1}{p} - \frac{1}{2}}\big\|\rho|u|^{p}\big\|_{L^1}^{\frac{1}{2} - \frac{1}{p}}\big\||u|^{\frac{p}{2} - 1}|\nabla u|\big\|_{L^2}.
\end{eqnarray}
Inserting the above into \eqref{1} and using Young's inequality, one has
\begin{eqnarray}\label{9}
\frac{d}{dt}\int \rho |u|^p dx + pc_0\int |u|^{p - 2}|\nabla u|^2 dx \lesssim \|\rho\|_{L^{p\gamma + 1 - \frac{p}{2}}}^{2\gamma + \frac{2}{p} - 1}\big\|\rho|u|^{p}\big\|_{L^1}^{1 - \frac{2}{p}}.
\end{eqnarray}
It follows from \eqref{UPB} that for $\frac{T}{2} \leq t < T$
\begin{eqnarray}\nonumber
\|\rho\|_{L^{p\gamma + 1 - \frac{p}{2}}}^{2\gamma + \frac{2}{p} - 1} &\lesssim& \|\rho^\gamma\|_{L^1}^{\frac{2}{p}}\|\rho\|_{L^\infty}^{\frac{2}{p}[(p - 1)\gamma + 1 - \frac{p}{2}]}\\\nonumber
&\lesssim& (T - t)^{- \frac{2}{p}[(p - 1)\gamma + 1 - \frac{p}{2}]\kappa}.
\end{eqnarray}
Due to \eqref{Constraint}, one has
\begin{equation}\nonumber
\frac{2}{p}[(p - 1)\gamma + 1 - \frac{p}{2}]\kappa <  \begin{cases}\frac{\gamma + (1 - \frac{2}{p})(\gamma - 1)}{\gamma + 3} < \frac{\gamma + \frac{1}{2}}{\gamma + 3},\quad 1 \leq \gamma \leq \frac{3}{2},
\\[-4mm]\\ \frac{\gamma + (1 - \frac{2}{p})(\gamma - 1)}{3\gamma} < \frac{2\gamma - 1}{3\gamma},\quad \gamma > \frac{3}{2}.
\end{cases}
\end{equation}
Hence,  $(T - t)^{- \frac{2}{p}[(p - 1)\gamma + 1 - \frac{p}{2}]\kappa}$ is integrable on $[0, T]$ and
\begin{eqnarray}\label{2}
\big\|\rho|u|^{p}\big\|_{L^1} \lesssim 1.
\end{eqnarray}
Using \eqref{2} and integrating \eqref{9} with respect to time, one further attains
$$\int_0^T\big\||u|^{\frac{p}{2} - 1}|\nabla u|\big\|_{L^2}^2dt \lesssim 1.$$ The proof of the lemma is completed.
\end{proof}

Recall the important quantity
$$\omega = u - v,$$
whose divergence is refereed to as the effective viscous flux in literature (see \cite{Hoff2} for instance).
Using \eqref{4} and Lemma \ref{energyes}, we have

\begin{lem}\label{fluxlem}
Let
\begin{equation}\nonumber
\delta = \begin{cases}
\frac{\gamma + 2}{\gamma + 4}  ,\quad 1 \leq \gamma \leq \frac{3}{2},\\[-4mm]\\
\frac{3\gamma - 1}{3\gamma + 1},\quad \gamma > \frac{3}{2}.
\end{cases}
\end{equation}
Under the assumptions of Theorem  \ref{main}, it holds that
$$\sup_{0 \leq t < T}\|\nabla \omega\|_{L^2} \lesssim 1,\quad \int_0^T\|\sqrt{\rho}\omega_t\|_{L^2}^2dt \lesssim 1,\quad \int_0^T\|\nabla^2\omega\|_{L^2}^{1 + \delta}dt \lesssim 1. $$
\end{lem}
\begin{proof}
Note that
\begin{equation}\nonumber
\begin{cases}
\rho\partial_t\omega - L\omega = \rho F,\\[-4mm]\\
\omega(t, \cdot)|_{\partial\Omega} = 0,
\end{cases}
\end{equation}
where $$F = - u\cdot\nabla u - A\partial_tL^{-1}\nabla\rho^\gamma.$$
A straightforward energy estimate gives that
\begin{eqnarray}\label{6}
&&\frac{1}{2}\frac{d}{dt}\int \big(\mu|\nabla \omega|^2 + (\lambda + \mu)|\nabla\cdot \omega|^2\big) dx + \int \rho|\partial_t\omega|^2\\\nonumber
&&= \int \rho F\omega_tdx \leq \frac{1}{2}\int \rho|\partial_t\omega|^2dx + \frac{1}{2}\int \rho|F|^2dx.
\end{eqnarray}
Clearly, one has
\begin{eqnarray}\nonumber
\frac{1}{2}\int \rho|F|^2dx \leq \int \rho|u\cdot\nabla u|^2dx + A^2\int \rho|\partial_tL^{-1}\nabla\rho^\gamma|^2dx.
\end{eqnarray}

For $1 \leq \gamma \leq \frac{3}{2}$, using interpolation inequalities, \eqref{4} and \eqref{UPB}, we can first estimate  that
\begin{eqnarray}\label{7}
&&\int \rho|\partial_tL^{-1}\nabla\rho^\gamma|^2dx\\\nonumber
&&= \int \rho |L^{-1}\nabla\nabla\cdot(\rho^\gamma u) + (\gamma - 1)L^{-1}\nabla(\rho^\gamma\nabla\cdot u)|^2 dx\\\nonumber
&&\lesssim \|\rho\|_{L^\infty} \|\rho^\gamma u\|_{L^2}^2 + \|\rho\|_{L^{\frac{3}{2}}}\|\rho^\gamma\nabla\cdot u\|_{L^2}^2 \\\nonumber
&&\lesssim \|\sqrt{\rho}u\|_{L^2}^{2}\|\rho\|_{L^\infty}^{2\gamma} +
\|\rho\|_{L^\gamma}^{\frac{2\gamma}{3}}\|\rho\|_{L^\infty}^{\frac{4\gamma}{3} + 1}\|\nabla u\|_{L^2}^2\\\nonumber
&&\lesssim (T - t)^{- 2\kappa \gamma} +
(T - t)^{- \kappa(\frac{7\gamma}{3} + 1)}\\\nonumber
&&\quad +\
(T - t)^{- \kappa(\frac{4\gamma}{3} + 1)}\|\nabla \omega\|_{L^2}^2.
\end{eqnarray}
For $\gamma > \frac{3}{2}$, one simply has
\begin{eqnarray}\label{7-1}
&&\int \rho|\partial_tL^{-1}\nabla\rho^\gamma|^2dx\\\nonumber
&&\lesssim \|\sqrt{\rho}u\|_{L^2}^{2}\|\rho\|_{L^\infty}^{2\gamma} +
\|\rho\|_{L^{\frac{3}{2}}}\|\rho\|_{L^\infty}^{2\gamma}\|\nabla u\|_{L^2}^2\\\nonumber
&&\lesssim (T - t)^{- 2\kappa \gamma} +
(T - t)^{- 3\kappa \gamma} +
(T - t)^{- 2\kappa \gamma}\|\nabla \omega\|_{L^2}^2.
\end{eqnarray}

Next, the second tern in \eqref{6} can be estimated as follow
\begin{eqnarray}\nonumber
\int \rho|u\cdot\nabla u|^2dx &\lesssim& \|\rho\|_{L^\infty}^{1 - \frac{2}{p}}\|\rho|u|^p\|_{L^1}^{\frac{2}{p}}\|\nabla u\|_{L^{\frac{2p}{p - 2}}}^2\\\nonumber
&\lesssim& (T - t)^{- \kappa(1 - \frac{2}{p})}\|\nabla v\|_{L^{\frac{2p}{p - 2}}}^2 + (T - t)^{- \kappa(1 - \frac{2}{p})}\|\nabla\omega\|_{L^2}^{2 - \frac{6}{p}}\|\nabla\omega\|_{L^6}^{ \frac{6}{p}}\\\nonumber
&\lesssim& (T - t)^{- \kappa(1 - \frac{2}{p})- \kappa\gamma(1 + \frac{2}{p})} + (T - t)^{- \kappa(1 - \frac{2}{p})}\|\nabla\omega\|_{L^2}^{2 - \frac{6}{p}}\|\nabla^2\omega\|_{L^2}^{ \frac{6}{p}}.
\end{eqnarray}
Here one has used interpolation inequality, \eqref{4} and Lemma \ref{energyes}. On the other hand, one also has
\begin{eqnarray}\nonumber
\|\nabla^2\omega\|_{L^2} &\lesssim& \|\rho\omega_t\|_{L^2} + \|\rho F\|_{L^2}\\\nonumber
&\lesssim& (T - t)^{-\frac{\kappa}{2}}\|\sqrt{\rho}\omega_t\|_{L^2} + (T - t)^{-\frac{\kappa}{2}}\|\sqrt{\rho} F\|_{L^2}.
\end{eqnarray}
Hence,
\begin{eqnarray}\label{8}
&&\int \rho|u\cdot\nabla u|^2dx\lesssim  (T - t)^{- 2\kappa  - \kappa(\gamma - 1)(1 + \frac{2}{p})}\\\nonumber
&&\quad +\ (T - t)^{- \kappa(1 + \frac{1}{p})}\|\nabla\omega\|_{L^2}^{2 - \frac{6}{p}}(\|\sqrt{\rho}\omega_t\|_{L^2} + \|\sqrt{\rho} F\|_{L^2})^{ \frac{6}{p}}.
\end{eqnarray}

By \eqref{7}, \eqref{7-1}, \eqref{8} and using Young's inequality, we have
\begin{eqnarray}\label{8-1}
\int \rho|F|^2dx &\leq& \frac{1}{2}\|\sqrt{\rho}\omega_t\|_{L^2}^2 + (T - t)^{- \kappa(1 + \frac{1}{p})\frac{p}{p - 3}}\|\nabla \omega\|_{L^2}^2\\\nonumber
&&+\ \Big\{(T - t)^{- \kappa(\frac{7\gamma}{3} + 1)} + (T - t)^{- \kappa(\frac{4\gamma}{3} + 1)}\|\nabla \omega\|_{L^2}^2\Big\}1_{1 \leq \gamma \leq \frac{3}{2}}\\\nonumber
&&+\ \Big\{(T - t)^{- 3\gamma\kappa}  + (T - t)^{- 2\gamma\kappa}\|\nabla \omega\|_{L^2}^2\Big\}1_{\gamma > \frac{3}{2}}.
\end{eqnarray}
Inserting the above into \eqref{6}, and using Gronwall's inequality and the integrability condition
\begin{equation}\nonumber
\begin{cases}
\kappa(1 + \frac{1}{p})\frac{p}{p - 3} < 1,\\[-4mm]\\
\kappa\big(\frac{7\gamma}{3} + 1\big) < \frac{7\gamma + 3}{3(\gamma + 3)}  < 1,\quad  {\rm if}\ 1 \leq \gamma \leq \frac{3}{2},\\[-4mm]\\
3\gamma\kappa < 1,\quad {\rm if}\ \gamma > \frac{3}{2},
\end{cases}
\end{equation}
one has
$$\sup_{0 \leq t < T}\|\nabla \omega\|_{L^2} \lesssim 1,\quad \int_0^T\big\|\sqrt{\rho}\omega_t\|_{L^2}^2dt \lesssim 1. $$
Consequently, using the above estimates and by revisiting \eqref{8-1}, one has
\begin{eqnarray}\nonumber
\int_0^T\int \rho|F|^2dxdt  \lesssim 1.
\end{eqnarray}
which in turn gives that
\begin{eqnarray}\nonumber
\int_0^t\|\nabla^2\omega\|_{L^2}^{1 + \delta}dt &\lesssim& \int_0^t\|\rho\|_{L^\infty}^{\frac{1 + \delta}{2}}\big(\|\sqrt{\rho}\omega_t\|_{L^2} + \|\sqrt{\rho} F\|_{L^2}\big)^{{1 + \delta}}ds\\\nonumber
&\lesssim& 1,
\end{eqnarray}
where one has used the Cauchy inequality and the integrability condition
\begin{equation}\nonumber
\frac{1 + \delta}{1 - \delta}\kappa <
\begin{cases}
\frac{1 + \frac{\gamma + 2}{\gamma + 4}}{1 - \frac{\gamma + 2}{\gamma + 4}}\frac{1}{\gamma + 3} = 1,\quad  {\rm if}\ 1 \leq \gamma \leq \frac{3}{2},\\[-4mm]\\
\frac{1 + \frac{3\gamma - 1}{3\gamma + 1}}{1 - \frac{3\gamma - 1}{3\gamma + 1}}\frac{1}{3\gamma} = 1,\quad  {\rm if}\ \gamma > \frac{3}{2}
\end{cases}
\end{equation}
This finishes the proof of the lemma.
\end{proof}

\begin{rem}
It should be emphasized here that the conclusion on the space-time estimate may not be true for $\delta = 1$.
But the scaling heuristics indicates that such kind of estimate with any $\delta > 0$ is crucial (but $\kappa$ may be smaller if $\delta$ is smaller).
\end{rem}

A straightforward consequence of Lemma \ref{fluxlem}, together with Lemma \ref{lem3} gives that
\begin{cor}\label{cor}
Let $\delta$ be given in Lemma \ref{fluxlem} and all other assumptions be the same as in Theorem \ref{main}. Then it holds that
$$\|\nabla u(t, \cdot)\|_{L^2}^2 \lesssim 1 + (T - t)^{-\kappa\gamma},\quad \int_0^T\|\nabla u\|_{L^6}^{1 + \delta}dt \lesssim 1. $$
\end{cor}
\begin{proof}
Indeed, one has
\begin{eqnarray}\nonumber
&&\|\nabla u(t, \cdot)\|_{L^2}^2 \lesssim \|\nabla v(t, \cdot)\|_{L^2}^2 + \|\nabla \omega(t, \cdot)\|_{L^2}^2\\\nonumber
&&\lesssim \|\rho^\gamma\|_{L^2}^2 + \|\nabla\omega\|_{L^2}^2 \lesssim (T - t)^{-\gamma\kappa} + \|\nabla\omega\|_{L^2}^2\\\nonumber
&&\quad\quad\quad \lesssim  1 + (T - t)^{-\gamma\kappa}
\end{eqnarray}
and
\begin{eqnarray}\nonumber
&&\int_0^T\|\nabla u\|_{L^6}^{1 + \delta}dt \lesssim \int_0^T\|\nabla v\|_{L^6}^{1 + \delta}dt + \int_0^T\|\nabla \omega\|_{L^6}^{1 + \delta}dt\\\nonumber
&&\lesssim \int_0^T\|\rho^\gamma\|_{L^6}^{1 + \delta}dt + \int_0^T\|\nabla \omega\|_{L^6}^{1 + \delta}dt \lesssim 1.
\end{eqnarray}
Here one has used the fact that $\frac{5\kappa\gamma(1 + \delta)}{6}  < 1$.
\end{proof}

\section{Further Estimate for the Effective Viscous Flux}

In this section we estimate the high order regularity of the quantity $\omega$. We follow the calculations in \cite{Hoff1}. So some similar calculations are omitted below.
Denote the material derivative of $f$ by
$$\dot{f} = \partial_tu + u\cdot\nabla u.$$
We have

\begin{lem}\label{lem4}
Let all assumptions in Theorem \ref{main} be true and $\delta$ be given in Lemma \ref{fluxlem}. Then one has
\begin{eqnarray}\label{53}
\int \rho|\dot{u}|^2dx + \int_0^t\|\nabla \dot{u}\|_{L^2}^2ds  \lesssim 1.
\end{eqnarray}
\end{lem}

\begin{proof}
Starting from
$$\rho\dot{u} + A\nabla \rho^\gamma = Lu,$$
one can derive that (see \cite{Hoff1})
\begin{eqnarray}\label{51}
&&\frac{d}{dt}\int \rho|\dot{u}|^2dx + \|\nabla \dot{u}\|_{L^2}^2 + \|\nabla \cdot \dot{u}\|_{L^2}^2\\\nonumber
&&\leq C\|\nabla u\|_{L^4}^4 + 2A \int\big[\partial_t\rho^\gamma \nabla\cdot\dot{u} + (u \cdot\nabla \dot{u})\cdot \nabla\rho^\gamma\big] dx.
\end{eqnarray}
First, note that
\begin{eqnarray}\nonumber
&&\int\partial_t\rho^\gamma \nabla\cdot\dot{u} + (u \cdot\nabla \dot{u})\cdot \nabla\rho^\gamma dx\\\nonumber
&&= - \int\big[\gamma\rho^\gamma(\nabla\cdot u)\nabla\cdot\dot{u} + (u\cdot\nabla\rho^\gamma) \nabla\cdot\dot{u} - (u \cdot\nabla \dot{u})\cdot \nabla\rho^\gamma\big] dx\\\nonumber
&&= - \int\big[(\gamma - 1)\rho^\gamma(\nabla\cdot u)\nabla\cdot\dot{u} + \rho^\gamma {\rm tr}(\nabla\dot{u}\nabla u)\big] dx\\\nonumber
&&\leq \|\rho^\gamma\|_{L^1}\|\rho\|_{L^\infty}^{3\gamma} + A^2\|\nabla u\|_{L^4}^4 + \frac{1}{4A}\|\nabla\dot{u}\|_{L^2}^2.
\end{eqnarray}
Inserting the above into \eqref{51} leads to
\begin{eqnarray}\label{52}
&&\frac{d}{dt}\int \rho|\dot{u}|^2dx + \|\nabla \dot{u}\|_{L^2}^2 + \|\nabla \cdot \dot{u}\|_{L^2}^2\\\nonumber
&&\leq C\|\nabla u\|_{L^4}^4 + C\|\rho^\gamma\|_{L^1}\|\rho\|_{L^\infty}^{3\gamma}\\\nonumber
&&\lesssim \|\nabla u\|_{L^4}^4 + (T - t)^{- 3\gamma\kappa}.
\end{eqnarray}
which, by integration with respect to time,  gives that
\begin{eqnarray}\label{52}
\int \rho|\dot{u}|^2dx + \int_0^t\|\nabla \dot{u}\|_{L^2}^2ds  \lesssim \int_0^t\|\nabla u\|_{L^4}^4ds + 1.
\end{eqnarray}

Now by the interpolation inequality and Corollary \ref{cor}, one has
\begin{eqnarray}\nonumber
\|\nabla u\|_{L^4}^4 &\leq& \|\nabla u\|_{L^2}\|\nabla u\|_{L^6}^3\\\nonumber
&\lesssim& \|\nabla u\|_{L^2}\|\nabla u\|_{L^6}\big(\|\nabla v\|_{L^6} + \|\nabla \omega\|_{L^6}\big)^2\\\nonumber
&\lesssim& \|\nabla u\|_{L^2}\|\nabla u\|_{L^6}\big(\|\rho^\gamma\|_{L^1}^{\frac{1}{6}}\|\rho\|_{L^\infty}^{\frac{5\gamma}{6}} + \|\nabla^2 \omega\|_{L^2}\big)^2\\\nonumber
&\lesssim& (T - t)^{- \frac{\gamma\kappa}{2} - \frac{5\gamma\kappa}{3}}\|\nabla u\|_{L^6} + (T - t)^{- \frac{\gamma\kappa}{2} - \kappa}\|\nabla u\|_{L^6}\|\sqrt{\rho}\dot{u}\|_{L^2}^2
\end{eqnarray}
Note that
\begin{eqnarray}\nonumber
&&(T - t)^{- \frac{\gamma\kappa}{2} - \frac{5\gamma\kappa}{3}}\|\nabla u\|_{L^6}\\\nonumber
&&\lesssim (T - t)^{- \frac{\gamma\kappa}{2} - \frac{5\gamma\kappa}{3}}\|\nabla v\|_{L^6} + (T - t)^{- \frac{\gamma\kappa}{2} - \frac{5\gamma\kappa}{3}}\|\nabla \omega\|_{L^6}\\\nonumber
&&\lesssim (T - t)^{- \frac{\gamma\kappa}{2} - \frac{5\gamma\kappa}{2}} + (T - t)^{- \frac{8\gamma\kappa}{3}}\|\sqrt{\rho}\dot{u}\|_{L^2}.
\end{eqnarray}
It is clear that $(T - t)^{- \frac{\gamma\kappa}{2} - \frac{5\gamma\kappa}{2}}$ is integrable in time.
Moreover, by the definition of $\delta$ and the constraint on $\kappa$, one also has
 $$\frac{\gamma\kappa}{2} + \kappa < \frac{\delta}{1 + \delta},\quad \frac{8\gamma\kappa}{3} < 1.$$
Hence, we have
$$\int_0^t(T - s)^{ - \frac{(\gamma + 2)\kappa}{2}}\|\nabla u\|_{L^6}ds \lesssim \Big(\int_0^T(T - t)^{ - \frac{\kappa(\gamma + 2)}{2}\frac{1 + \delta}{\delta}}dt\Big)^{\frac{\delta}{1 + \delta}} \lesssim 1.$$
Using \eqref{52} and Gronwall's inequality, one can finish the proof of the lemma.
\end{proof}

\section{Blowup Criterions and Proof of Theorem \ref{main}}

Consider the continuity equation. Applying $\nabla$ and then performing the energy estimate, one can  derive easily that
\begin{eqnarray}\nonumber
\frac{d}{dt}\|\nabla \rho\|_{L^q}^q \lesssim \|\nabla u\|_{L^\infty}\|\nabla \rho\|_{L^q}^q + \int \rho|\nabla^2u||\nabla \rho|^{q - 1}dx.
\end{eqnarray}
For $3 < q < 6$, one can estimate the last term in the above inequality by
\begin{eqnarray}\nonumber
&&\int \rho|\nabla^2u||\nabla \rho|^{q - 1}dx\\\nonumber
&&\lesssim \|\rho\|_{L^\infty}\big(\|\nabla^2v\|_{L^q} + \|\nabla^2\omega\|_{L^q}\big)\|\nabla \rho\|_{L^q}^{q - 1}\\\nonumber
&&\lesssim \|\rho\|_{L^\infty}^\gamma\|\nabla \rho\|_{L^q}^{q} + \|\rho\|_{L^\infty}\|\nabla^2\omega\|_{L^q}\|\nabla \rho\|_{L^q}^{q - 1},
\end{eqnarray}
where \eqref{5} has been used. Moreover, by Lemma \ref{lem1}, there holds
\begin{eqnarray}\nonumber
\|\nabla u\|_{L^\infty} &\lesssim& \|\nabla\omega\|_{L^\infty} + \|\nabla v\|_{L^\infty}\\\nonumber
&\lesssim& 1 + \|\nabla\omega\|_{L^\infty} + \|\rho\|_{L^\infty}^\gamma\ln\big(e + \|\rho\|_{L^\infty}^{\gamma - 1}\|\nabla\rho\|_{L^q}\big).
\end{eqnarray}
Combining the above three estimates together, we arrive at
\begin{eqnarray}\nonumber
\frac{d}{dt}\|\nabla \rho\|_{L^q} &\lesssim& \|\rho\|_{L^\infty}\|\nabla^2\omega\|_{L^q} + \big(1 + \|\nabla \omega\|_{L^\infty} + \|\rho\|_{L^\infty}^\gamma\ln(e + \|\rho\|_{L^\infty}^{\gamma - 1})\big)\\\nonumber
&& \times\ \|\nabla \rho\|_{L^q}\ln\big(e + \|\nabla\rho\|_{L^q}\big)\\\nonumber
&\lesssim& \big(1 + \|\nabla \omega\|_{L^\infty} + \|\rho\|_{L^\infty}\|\nabla^2\omega\|_{L^q} + \|\rho\|_{L^\infty}^\gamma\ln(e + \|\rho\|_{L^\infty})\big)\\\nonumber
&&\times\ \|\nabla \rho\|_{L^q}\ln\big(e + \|\nabla\rho\|_{L^q}\big).
\end{eqnarray}
Then the Gronwall's inequality shows that
$$\|\nabla \rho\|_{L^q} \lesssim 1,\quad 0 \leq t \leq T$$
provided that
\begin{equation}\label{21}
\begin{cases}
\int_0^T\|\rho\|_{L^\infty}^\gamma \ln(e + \|\rho\|_{L^\infty}) dt < \infty,\\[-4mm]\\
\int_0^T\big(\|\nabla \omega\|_{L^\infty} + \|\rho\|_{L^\infty}\|\nabla^2\omega\|_{L^q}\big) dt < \infty.
\end{cases}
\end{equation}

Clearly, the first one in \eqref{21} holds under \eqref{UPB}.
It remains to check the second inequality in \eqref{21}. By interpolation and using Lemma \ref{fluxlem}, one has
\begin{eqnarray}\nonumber
&&\|\nabla \omega\|_{L^\infty} + \|\rho\|_{L^\infty}\|\nabla^2\omega\|_{L^q}\\\nonumber
&& \lesssim \|\nabla\omega\|_{L^2}^{\frac{2q - 6}{5q - 6}}\|\nabla^2\omega\|_{L^q}^{\frac{3q}{5q - 6}} + \|\rho\|_{L^\infty}\|\nabla^2\omega\|_{L^q}\\\nonumber
&&\lesssim (T - t)^{- \kappa}\|\nabla^2\omega\|_{L^q}.
\end{eqnarray}
Hence, using
$$\rho\dot{u} = L\omega,$$
one can get
\begin{eqnarray}\nonumber
&&\int_0^T\big(\|\nabla \omega\|_{L^\infty} + \|\rho\|_{L^\infty}\|\nabla^2\omega\|_{L^q}\big) dt\\\nonumber
&&\lesssim \int_0^T (T - t)^{- \kappa}\|\rho\dot{u}\|_{L^q}ds\\\nonumber
&&\lesssim \int_0^T (T - t)^{- \kappa}\|\rho\|_{L^{\frac{6q}{6 - q}}}\|\dot{u}\|_{L^6}ds.
\end{eqnarray}
For $\gamma \leq \frac{6q}{6 - q}$, then one can use the interpolation inequality to derive that
\begin{eqnarray}\nonumber
&&\int_0^T\big(\|\nabla \omega\|_{L^\infty} + \|\rho\|_{L^\infty}\|\nabla^2\omega\|_{L^q}\big) dt\\\nonumber
&&\lesssim \int_0^T (T - t)^{- \kappa(2 - \frac{6 - q}{6q}\gamma)}\|\rho^\gamma\|_{L^1}^{\frac{6 - q}{6q}}\|\nabla\dot{u}\|_{L^2}ds\\\nonumber
&&\lesssim \int_0^T\|\nabla\dot{u}\|_{L^2}^2dt + \int_0^T(T - t)^{- \kappa(4 - \frac{6 - q}{3q}\gamma)}dt.
\end{eqnarray}
For $\gamma > \frac{6q}{6 - q}$, it holds that
\begin{eqnarray}\nonumber
&&\int_0^T\big(\|\nabla \omega\|_{L^\infty} + \|\rho\|_{L^\infty}\|\nabla^2\omega\|_{L^q}\big) dt\\\nonumber
&&\lesssim \int_0^T\|\nabla\dot{u}\|_{L^2}^2dt + \int_0^T(T - t)^{- 2\kappa}dt.
\end{eqnarray}
Using \eqref{53} in Lemma \ref{lem4} and noting that $\kappa(4 - \frac{6 - q}{3q}\gamma) < 1$, one has
\begin{eqnarray}\nonumber
\int_0^T\big(\|\nabla \omega\|_{L^\infty} + \|\rho\|_{L^\infty}\|\nabla^2\omega\|_{L^q}\big) dt \lesssim 1.
\end{eqnarray}

Now we have proved that
$$\|\nabla \rho\|_{L^q} \lesssim 1,\quad 0 \leq t \leq T.$$
By Sobolev imbedding, one has
$$\rho \lesssim 1, \quad 0 \leq t < T.$$
Now the proof of Theorem \ref{main} follows from the known blowup criteria, see \cite{HLX2, SWZ}. Alternatively, one can also estimate that
\begin{eqnarray}\nonumber
\|\nabla u(t, \cdot)\|_{L^2}^2 &\lesssim& \|\nabla v(t, \cdot)\|_{L^2}^2 + \|\nabla \omega(t, \cdot)\|_{L^2}^2\\\nonumber
&\lesssim& \|\rho^\gamma\|_{L^2}^2 + \|\nabla\omega\|_{L^2}^2 \lesssim 1
\end{eqnarray}
for all $0 \leq t < T$. Then using non-blowup criterion \eqref{criterion}, one finishes the proof of Theorem \ref{main}.

\section{Construction of the Explicit Blowup Solution}

Now we construct the explicit blowup solution and prove Theorem \ref{Exam}. The key here is to find an explicit solution to an over-determined nonlinear system of equations which serves as  the profile of a self-similar solution to the compressible Euler and Navier-Stokes equations.

Let us derive the system governing the profile $(\Theta, V)$. Denote
$$y = \frac{x}{(T - t)^{\frac{\gamma + 1}{2\gamma}}}.$$ Due to \eqref{SSS}, it is easy to compute that
$$\partial_t\rho(t, x) = \frac{1}{(T - t)^{\frac{1}{\gamma} + 1}}\big[\frac{1}{\gamma}\Theta(y) + \frac{\gamma + 1}{2\gamma}y\cdot\nabla_y\Theta(y)\big]$$
and
$$\rho\partial_tu (t, x) = \frac{1}{(T - t)^{\frac{\gamma + 1}{2\gamma} + 1}}\big[\frac{\gamma - 1}{2\gamma}V(y) + \frac{\gamma + 1}{2\gamma}y\cdot\nabla_y V(y)\big]\Theta(y).$$
Moreover, one can also compute that
$$\nabla_x\cdot(\rho(t, x) u(t, x)) = \frac{1}{(T - t)^{\frac{1}{\gamma} + 1}}\nabla_y \cdot (\Theta(y) V(y))$$
and
$$\rho u\cdot\nabla_x u + A\nabla_x\rho^\gamma = \frac{1}{(T - t)^{\frac{\gamma + 1}{2\gamma} + 1}}\big[\Theta V\cdot\nabla_y V(y) + A\nabla_y\Theta^\gamma(y)\big],$$
$$\mu\Delta_xu + (\lambda + \mu)\nabla_x\nabla_x\cdot u = \frac{1}{(T - t)^{\frac{\gamma + 1}{2\gamma} + 1}}\big(\mu\Delta_yV(y) + (\lambda + \mu)\nabla_y\nabla_y\cdot V(y)\big).$$
Consequently, we obtain the systems governing the profile $(\Theta, V)$
\begin{equation}\label{SCNS}
\begin{cases}
\frac{1}{\gamma}\Theta + \frac{\gamma + 1}{2\gamma}y\cdot\nabla \Theta + \nabla\cdot (\Theta V) = 0,\\[-4mm]\\
\big(\frac{\gamma - 1}{2\gamma}V + \frac{\gamma + 1}{2\gamma}y\cdot\nabla V +  V\cdot\nabla V\big)\Theta + A\nabla \Theta^\gamma
  = \mu\Delta V + (\lambda + \mu)\nabla\nabla\cdot V.
\end{cases}
\end{equation}

Next, let us construct a special solution to \eqref{SCNS}. Here we hope to ignore the viscosity terms in \eqref{SCNS}.  Hence, we search for solutions to \eqref{SCNS} with
$$V(y) = \beta y$$
for some constant $\beta$ which will be determined later. Using this linear velocity field $V$, we reduce \eqref{SCNS} to
\begin{equation}\label{SCE}
\begin{cases}
\big(\frac{1}{\gamma} + n\beta\big)\Theta + \big(\frac{\gamma + 1}{2\gamma} + \beta\big)y\cdot\nabla \Theta = 0,\\[-4mm]\\
A\nabla \Theta^\gamma = - (1 + \beta)\beta y\Theta,\\[-4mm]\\
  V(y) = \beta y.
\end{cases}
\end{equation}
Clearly, system \eqref{SCE} is over-determined. Fortunately, by choosing $\beta$ so that
$$\beta = - \frac{2}{n(\gamma - 1) + 2},$$
one has a solution to \eqref{SCE}, which reads
\begin{equation}\label{Solu}
\begin{cases}
\Theta(y) = C_n^{\frac{1}{\gamma - 1}}r^{\frac{2}{\gamma - 1}},\\[-4mm]\\
  V(y) = - \frac{2y}{n(\gamma - 1) + 2},
\end{cases}
\end{equation}
where the constants $C_n$ are given in \eqref{Const}.
It is clear that the example in Theorem \ref{Exam} can be obtained by inserting \eqref{Solu} into \eqref{SSS}.

\begin{rem}
It is interesting to construct regular solutions $(\Theta, V)$ to system \eqref{SCNS} with
\begin{equation}\nonumber
\nabla V \in H^s(s \geq 2),\quad \Theta \in L^p \cap C^\infty.
\end{equation}
A solution in this class would yield a family of physically more reasonable solutions which blow up in any given finite time. Unfortunately, we are not able to settle this problem down at present. However, using the first equation in \eqref{SCNS}, it is easy to derive the following \textit{a priori} estimate
\begin{eqnarray}\nonumber
\Big|\frac{2p - 3(\gamma + 1)}{2p\gamma}\Big|\|\Theta\|_{L^p}^p \leq \frac{p - 1}{p}\|\Theta\|_{L^p}^p\|\nabla\cdot V\|_{L^\infty}.
\end{eqnarray}
This simply implies that profiles $(\Theta, V)$ with sufficiently small $\|\Theta\|_{L^p}$ and $\|\nabla V\|_{H^2}$ do not exist for any $p \neq \frac{3(\gamma + 1)}{2}$. Hence, there is no self-similar blowup solutions to the compressible Navier-Stokes equations \eqref{CNS} with sufficiently small initial data $\int\rho_0|u_0|^2$ ($\lesssim \|\rho_0\|_{L^{\frac{3}{2}}}\|\nabla u_0\|_{L^2}^2$), which agrees with the result in \cite{HLX}.
\end{rem}

\section*{Acknowledgement}
Zhen Lei was in part supported by NSFC (grant No. 11421061 and 11222107),
National Support Program for Young Top-Notch Talents, and SGST 09DZ2272900.
Zhouping Xin was partially supported by Zheng Ge Ru Foundation, Hong Kong RGC Earmarked
Research Grants CUHK-14305315 and CUHK4048/13P, NSFC/RGC Joint Research
Scheme Grant N-CUHK 443-14, and a Focus Area Grant from the Chinese University of Hong
Kong.


\begin{thebibliography}{999}



\bibitem{BW} 5.H. Br${\rm \acute{e}}$zis, S. Wainger, \textit{A note on limiting cases of Sobolev embeddings and convolution inequalities}, Comm. Partial Differential Equations 5 (1980) 773--789.
\bibitem{CKN:1}
 L. Caffarelli, R. Kohn, and L. Nierenberg,
 {\em Partial regularity of suitable weak solutions
 of the Navier-Stokes equations},
 Comm. Pure Appl. Math., 35 (1982),  771--831.

 \bibitem{Chemin} J. M. Chemin,  \textit{Remarques sur l'sexistence globale pour le syst${\rm \grave{e}}$me
de Navier-CStokes incompressible}. SIAM Journal on Mathematical
Analysis 1992; 23:20--28.

\bibitem{CSTY} C. C. Chen, R. Strain, T. P. Tsai, H. T. Yau, \textit{Lower bound on the blow-up rate of the axisymmetric Navier每Stokes equations}, Int. Math. Res. Not. IMRN 9 (2008), Art. ID rnn016, 31

\bibitem{CSYT} C. C. Chen, R. Strain, T. P. Tsai, H. T. Yau, \textit{Lower bounds on the blow-up rate of the axisymmetric Navier每Stokes equations. II}, Comm. Partial Differential Equations 34 (1-3) (2009) 203--232.

\bibitem{CMZ} Q. Chen, C. Miao and Z. Zhang, \textit{Global well-posedness for the compressible Navier-
Stokes equations with the highly oscillating initial velocity}, to appear in Comm. Pure.
Appl. Math.

\bibitem{CCK-1} Cho, Y.; Choe, H. J.; Kim, H. \textit{Unique solvability of the initial boundary value
problems for compressible viscous fluid}. J. Math. Pures Appl. 83 (2004), 243--275.

\bibitem{CCK-2} Choe, H. J.; Kim, H. Strong solutions of the Navier-Stokes equations for isentropic
compressible fluids. J. Differ. Eqs. 190 (2003), 504--523.

\bibitem{Danchin} R. Danchin, \textit{Global existence in critical spaces for compressible Navier-Stokes equations},
Invent.Math., 141(2000), 579--614.

\bibitem{Des} B. Desjardins, \textit{Regularity of weak solutions of the compressible isentropic Navier-
Stokes equations}, Comm. P. D. E., 22(1997), 977-1008.

\bibitem{FJO} J. Fan, S. Jiang and Y. Ou, \textit{A blow-up criterion for the compressible viscous heat-conductive
flows}, Annales de l'Institut Henri Poncar${\rm \acute{e}}$-Analyse non lin${\rm \acute{e}}$aire, online.

\bibitem{Fei1} Feireisl, E. \textit{Dynamics of viscous compressible fluids}. Oxford University Press, New
York, 2004.

\bibitem{Fei2} Feireisl, E.; Novotny, A.; Petzeltov${\rm \acute{a}}$, H. \textit{On the existence of globally defined weak
solutions to the Navier-Stokes equations}. J. Math. Fluid Mech. 3 (2001), no. 4,
358-392.

\bibitem{FK} H. Fujita and T. Kato  \textit{On the Navier-Stokes initial value problem I}.
Archive for Rational Mechanics and Analysis 1964; 16:269--315.


\bibitem{Hoff1} Hoff, D. \textit{Global solutions of the Navier-Stokes equations for multidimensional compressible
flow with discontinuous initial data}. J. Differ. Eqs. 120 (1995), no. 1,
215--254.

\bibitem{Hoff2} Hoff, D. \textit{Strong convergence to global solutions for multidimensional flows of compressible,
viscous fluids with polytropic equations of state and discontinuous initial
data}. Arch. Rational Mech. Anal. 132 (1995), 1--14.

\bibitem{Hoff3} Hoff, D. \textit{Dynamics of singularity surfaces for compressible, viscous flows in two
space dimensions}. Comm. Pure Appl. Math. 55(2002), no. 11, 1365--1407.

\bibitem{HLX} Huang, X. D.; Li, J.; Xin, Z. P. \textit{Global well-posedness of classical solutions with
large oscillations and vacuum to the three-dimensional isentropic compressible
Navier-Stokes equations}. Comm. Pure Appl. Math. 65, 549--585 (2012)

\bibitem{HZ1} Huang, X.;, Li, J.; Xin, Z. P., \textit{Blowup criterion for viscous baratropic flows with vacuum states}.
Comm. Math. Phys. 301 (2011), no. 1, 23--35.

\bibitem{HLX2} Huang, X. D.; Li, J.; Xin, Z. P., \textit{Serrin-type criterion for the three-dimensional viscous compressible flows}. SIAM J. Math. Anal. 43 (2011), no. 4, 1872--1886.

\bibitem{HZ2} X. Huang and Z. Xin, A blow-up criterion for classical solutions to the compressible Navier-Stokes equations. Sci. China Math. 53 (2010), no. 3, 671--686.

\bibitem{HXin} X. Huang and Z. Xin, \textit{On formation of singularity for non-isentropic
Navier-Stokes equations without heat-conductivity}, arXiv:1501.06291.

\bibitem{Itaya} N. Itaya, \textit{On the Cauchy problem for the system of fundamental equations describing the movement of compressible
viscous fluids}, Kodai Math. Sem. Rep. 23 (1971) 60--120.

\bibitem{JW} L. Jiang and Y. Wang, On\textit{ the blow up criterion to the 2-D compressible Navier-Stokes
equations}, Preprint, 2009.

\bibitem{JZ1} S. Jiang and P. Zhang, \textit{Global spherically symmetric solutions of the compressible
isentropic Navier-Stokes equations}, Comm. Math. Phys., 215(2001), 559--581.

\bibitem{JZ2} S. Jiang and P. Zhang, \textit{Axisymmetric solutions of the 3-D Navier-Stokes equations
for compressible isentropic flows}, J. Math. Pure Appl., 82(2003), 949--973.

\bibitem{Kato} T. Kato,  \textit{Strong $L^p$ solutions of the Navier-Stokes
equations in $\mathbb{R}^m$, with applications to weak solutions}.
Mathematische Zeitschrift 1984; 187:471--480.

\bibitem{KTa} H. Koch and D. Tataru,  \textit{Well-posedness for the Navier-Stokes
equations}. Advance Mathematics 2001; 157:22--35.

\bibitem{KNSS} G. Koch, N. Nadirashvili, G. Seregin, V. Sverak, \textit{Liouville theorems for the Navier每Stokes equations and applications}, Acta Math. 203 (1) (2009) 83--105.

\bibitem{KT} H. Kozono, Y. Taniuchi, \textit{Limiting case of the Sobolev inequality in BMO}, with application to the Euler equations, Comm. Math. Phys. 214 (2000) 191--200.

\bibitem{LL}  Z. Lei and F.H. Lin, \textit{Global mild solutions of Navier-Stokes equations}, Comm. Pure
Appl. Math. 64 (2011), no. 9, 1297--1304

\bibitem{LDZ} Z. Lei, Y. Du and Q. T. Zhang, \textit{Singularities of solutions to compressible Euler equations with vacuum}, Math. Res. Lett. 20 (2013), no. 1, 41--50.

\bibitem{LZ} Z. Lei and Q. Zhang, \textit{A Liouville theorem for the axially-symmetric Navier-Stokes equations}. J. Funct. Anal. 261 (2011), no. 8, 2323--2345.

\bibitem{LWX} H. Li, X. Wang and Z. Xin, \textit{Non-existence of classical solutions with finite energy to the Cauchy problem of the compressible Navier-Stokes equations}, preprint.

\bibitem{Lions} Lions, P. L. \textit{Mathematical topics in fluid mechanics}. Vol. 2. Compressible models.
Oxford University Press, New York, 1998.

\bibitem{Luo} Z. Luo. Global existence of classical solutions to two-dimensional Navier-Stokes
equations with Cauchy data containing vacuum, Math. Methods Appl. Sci., in
press. DOI: 10.1002/mma.2896.

\bibitem{MN} Matsumura, A.; Nishida, T. \textit{The initial value problem for the equations of motion
of viscous and heat-conductive gases}. J. Math. Kyoto Univ. 20(1980), no. 1, 67--104.

\bibitem{MN2} A. Matsumura and T. Nishida, The initial boundary value problems for the equations
of motion of compressible and heat-conductive fluids, Comm. Math. Phys., 89(1983)
445-464.

\bibitem{Nash} Nash, \textit{J. Le probl${\rm \grave{e}}$me de Cauchy pour les ${\rm \grave{e}}$quations diff${\rm \grave{e}}$rentielles d'un fluide
g${\rm \grave{e}}$ral}. Bull. Soc. Math. France. 90 (1962), 487--497.

\bibitem{Planchon} F. Planchon,  \textit{Global strong solutions in Sobolev or Lebesgue spaces
to the incompressible Navier-Stokes equations in $\mathbb{R}^3$}.
Annales l'Institut Henri Poincare 1996; 13:319--336.

\bibitem{Serrin} Serrin, J. \textit{On the uniqueness of compressible fluid motion}. Arch. Rational. Mech.
Anal. 3 (1959), 271-288.

\bibitem{Sideris} T. C. Sideris, \textit{Formation of singularities in three-dimensional compressible fluids}. Comm. Math. Phys.
101 (1985), no. 4, 475--485.

\bibitem{Solo} V. A. Solonnikov, \textit{The solvability of the initial-boundary value problem for the equations
of motion of a viscous compressible fluid}, J. Sov. Math., 14 (1980), 1120--1133.

\bibitem{SWZ} Y. Z. Sun, C. Wang and Z. F. Zhang, \textit{A Beale每Kato每Majda blow-up criterion for the 3-D compressible Navier每Stokes equations}, J. Math. Pure Appl., 95 (2011), 36--47.

\bibitem{WZ} H. Wen and C. Zhu, \textit{Blow-up criterions of strong solutions to 3D compressible Navier-Stokes equations with vacuum}. Adv. Math. 248 (2013), 534--572.

\bibitem{Xin} Xin, Z. P. \textit{Blowup of smooth solutions to the compressible Navier-Stokes equation
with compact density}. Comm. Pure Appl. Math. 51 (1998), 229--240.

\bibitem{XY} Xin, Z. P.; Yan, W. \textit{On blowup of classical solutions to the compressible Navier-
Stokes equations}. Comm. Math. Phys. 321 (2013), no. 2, 529--541.

\end{thebibliography}
\end{document}